\documentclass{amsart}
\usepackage[utf8]{inputenc}
\usepackage[english]{babel}
\usepackage{amssymb}
\usepackage{amsmath}
\usepackage{amsthm}
\usepackage[left=3.5cm, right=3.5cm, top=2cm, bottom=3cm]{geometry}
\usepackage{mathtools}
\usepackage{xcolor}
\usepackage{tikz-cd}

\newcommand{\gll}{\mathfrak{gl}_2}
\newcommand{\gln}{\mathfrak{gl}_n}
\newcommand{\yt}{\mathrm Y(\mathfrak{gl}_2)}
\newcommand{\sll}{\mathfrak{sl}_2}
\newcommand{\Dab}{M(\underline a,\underline b)}
\newcommand{\Lab}{L(\underline a,\underline b)}
\newcommand{\Labop}{L(\underline a^{op},\underline b^{op})}

\newtheorem{proposition}{Proposition}
\newtheorem{lemma}{Lemma}
\newtheorem{theorem}{Theorem}
\newtheorem{corol}{Corollary}

\newtheorem*{remark}{Remark}
\newtheorem{lettertheorem}{Theorem}

\theoremstyle{definition}
\newtheorem*{definition}{Definition}

\title{Simplicity of spectra for Bethe subalgebras in $\yt$}
\author{Inna Mashanova-Golikova}

\begin{document}
\maketitle
\begin{abstract}
    We consider Bethe subalgebras B(C) in the Yangian $\yt$ with $C$ regular $2\times 2$ matrix. We study the action of Bethe subalgebras of $\yt$ on finite-dimensional representations of $\yt$. We prove that $B(C)$ with real diagonal $C$ has simple spectrum on any irreducible $\yt$-module corresponding to a disjoint union of real strings. We extend this result to limits of Bethe algebras. Our main tool is the computation of Shapovalov-type determinant for the nilpotent degeneration of $B(C)$.
\end{abstract}

\section{Introduction}
\subsection{Yangians and Bethe subalgebras}
The Yangian $Y(\gln)$ is a Hopf algebra deformation of the enveloping algebra $U(\gln[z])$ of $\gln[z]$, the Lie algebra of polynomial maps $\mathbb C\to\gln$. It is one of the first examples of quantum groups. This algebra was considered in the works of L. Fadeev and St.-Petersburg school in relation with the inverse scattering method, see e.g. \cite{T,TF}. We refer the reader to \cite{M} for more details.

The algebra $Y(\gln)$ is generated by elements $t_{ij}^{(r)}$, $1\leqslant i,j\leqslant n$, $r\in \mathbb Z_{\geqslant 0}$ and $t_{ij}^{(0)}=\delta_{ij}$. (The elements $t_{ij}^{(r)}$ correspond to $E_{ij}z^r\in\gln[z]$ where $E_{ij}\in\gln$ is the standard matrix unit.) The relations are \[[t_{ij}^{(r+1)},t_{kl}^{(s)}]-[t_{ij}^{(r)},t_{kl}^{(s+1)}]=t_{kj}^{(r)}t_{il}^{(s)}-t_{kj}^{(s)}t_{il}^{(r)}.\]
Introduce the formal power series in $u^{-1}$, where $u$ is a formal variable, $$t_{ij}(u)=\sum_{r\geqslant 0}t_{ij}^{(r)}u^{-r}.$$
These formal power series can be combined into a matrix with values in formal series with coefficients in $Y(\gln)$
$$T(u)=\sum_{i,j}e_{ij}\otimes t_{ij}(u)\in \mathrm{End}(\mathbb C^n)\otimes Y(\gln)[[u^{-1}]],$$
where $e_{ij}$ is the standard matrix unit.

The Yangian $Y(\gln)$ has a filtration generated by $\deg t_{ij}^{(r)}=r$. It follows from the relations that this indeed gives a filtration on $Y(\gln)$.

In this text we are considering the case $n=2$.

It is of interest to study the Bethe subalgebras $B(C)$, a family of commutative subalgebras of $\yt$ parametrized by complex matrices $C\in Mat_2$. These subalgebras come from studying the {\em XXX} Heisenberg model. The Hamiltonian of the {\em XXX} chain with external magnetic field is the image of some element from $B(C)$ in the tensor product of evaluation representations of the Yangian $(\mathbb C^2)^{\otimes n}$. The generators of $B(C)$ then give a complete set of integrals of the XXX chain.

The algebra $B(C)$ is generated by all the coefficients of two following formal power series, $$\operatorname{qdet}T(u)=t_{11}(u)t_{22}(u-1)-t_{21}(u)t_{12}(u-1)$$ and $$\operatorname{tr}CT(u)=c_{11}t_{11}(u)+c_{12}t_{21}(u)+c_{21}t_{12}(u)+c_{22}t_{22}(u).$$
The algebra does not change under dilations of $C$, hence the family is parametrized by points in $\mathbb CP^3=\mathbb{P}(Mat_2)$.

\subsection{Limits of Bethe subalgebras}
If $C$ is a regular matrix, then $B(C)$ is a maximal commutative subalgebra of $\yt$, as shown in \cite{NO}. More precisely, for a regular $C$, all the coefficients of $\operatorname{qdet}T(u)$ and $\operatorname{tr}CT(u)$ generate $B(C)$ and are algebraically independent. The Poincare series of $B(C)$ with respect to the above filtration is $$P_{B(C)}(t)=\prod_{k=1}^{\infty}\frac{1}{(1-t^k)^2}.$$
For non-regular $C$ the Poincare series drops. Namely, the coefficients at $u^{-1}$ of $\operatorname{tr}CT(u)$ and $\operatorname{qdet}T(u)$ are both equal to $t_{11}^{(1)}+t_{22}^{(1)}$, while all other coefficients remain algebraically independent, and the Poincare series is equal to
$$P_{B (\left (\begin{smallmatrix} 1 & 0\\ 0 & 1 \end{smallmatrix}\right ) )}(t)=\frac{1}{1-t}\prod_{k=2}^{\infty}\frac{1}{(1-t^k)^2}.$$

This and the more general situation of $\gln$ have been studied in \cite{IR1}.

We study maximal commutative subalgebras, in the sense of \cite{NO}, so following \cite{IR1} we complete this smaller subalgebra to have the same Poincare series as for generic $C$. This completion is defined as the limit $$\lim_{t\to 0}B(\left (\begin{smallmatrix} 1 & 0\\ 0 & 1 \end{smallmatrix}\right )+tC')$$ and depends on the choice of direction in $\mathbb CP^3$, i.e., $C'$. We consider the family of Bethe subalgebras $B(C)$ for regular $C\in Mat_2$ and define its closure. We prove the following result:

\begin{lettertheorem}
\label{introsubalgfam}
The closure $\mathcal B$ of the family of Bethe subalgebras in $\yt$ is parametrized by the points of the blow up of $\mathbb CP^3$ at the point 
corresponding to $\left (\begin{smallmatrix} 1 & 0\\ 0 & 1 \end{smallmatrix}\right )$.

\end{lettertheorem}

We denote by $Z$ this blow-up of $\mathbb CP^3$, i.e., the parameter space of the family $\mathcal B$.

\begin{remark}
The family $\mathcal B$ is a flat family of commutative subalgebras of $\yt$ over $Z$. In \cite{IR2} the definition of $B(C)$ is extended to the points of the De Concini-Procesi compactification of the adjoint Lie group of the Lie algebra for which the Yangian is defined. And it is expected that the limit space is some resolution of the De Concini-Procesi compactification. In our case the algebra is $\gll$ and the corresponding group is $PGL(2,\mathbb C)$. Its De Concini-Procesi compactification is $\mathbb CP^3$.
\end{remark}

\subsection{Representations of $\yt$ and the {\em XXX} chain.}
One can define the action of $\yt$ on a representation of $\gll$ using the evaluation morphism $\yt\to U(\gll)$, $$t_{ij}(u)\mapsto \delta_{ij}+E_{ij}u^{-1}.$$ We will call these representations of $\yt$ the evaluation representations. Since $\yt$ is a Hopf algebra, it also acts on tensor products of the evaluation representations. Any finite-dimensional irreducible representation of $\yt$ is isomorphic to a submodule or a quotient of a tensor product of evaluation representations with respect to the Hopf algebra structure on $\yt$.

We consider the action of $B(C)$ on finite-dimensional $\yt$ modules. These modules are the state spaces for the {\em XXX} chain and the conservation laws are the elements of $B(C)$.

Let $L(a,b)$ denote the evaluation representation of $\yt$ which comes from the finite-dimensional representation of $\gll$ with highest weight $(a,b)$. Then $\mathcal B(x)$ acts on $\Lab=L(a_1,b_1)\otimes\ldots\otimes L(a_n,b_n)$ for any $x\in Z$.

In \cite{MTV2} it has been shown that for representations $L(a_1,a_1-1)\otimes\ldots\otimes L(a_n,a_n-1)$, where $a_1,\ldots,a_n$ are generic, the eigenspaces for the action of $B(\left (\begin{smallmatrix} 1 & 0\\ 0 & 1 \end{smallmatrix}\right ))$ are irreducible $\gll$ submodules.

\subsection{Main results}
We extend the results of \cite{MTV2} as follows.

A string is a set $S(a,b)=\{a-1,a-2,\ldots,b+1,b\}$ for $a,b\in\mathbb C$, $a>b$. It is known that the representation $\Lab$ is irreducible if and only if, for any $1\leqslant i<j\leqslant n$, one of three possibilities hold: $S(a_i,b_i)\cup S(a_j,b_j)$ is not a string, or $S(a_i,b_i)\subset S(a_j,b_j)$, or $S(a_i,b_i)\supset S(a_j,b_j)$.

In the paper we prove the following two results:

\begin{lettertheorem}
\label{introcyclvec}
The action of any algebra in the family $\mathcal B$ in $L(a_1,b_1)\otimes\ldots\otimes L(a_n,b_n)$ has a cyclic vector, if, for any $1\leqslant i<j\leqslant n$, $S(a_i,b_i)\cup S(a_j,b_j)$ is not a string.
\end{lettertheorem}

We reduce this statement to the case of the principal nilpotent $C=e_{12}$ using that any non-scalar matrix in $\gll$ can be taken to any open neighbourhood of the principal nilponent by conjugation and dilation. Then we use that the condition of having a cyclic vector is open.

The case of the principal nilpotent is treated by proving that in the tensor product of the corresponding $\gll$ Verma modules as representations of $\yt$, the product of highest weight vectors is cyclic for $B(e_{12})$.

Secondly, we restrict to the closure of the subfamily corresponding to real diagonal matrices parametrized by the points of $\mathbb RP^1\simeq Z'\subset Z$.

\begin{lettertheorem}
\label{introsimpspec}
For any $x\in Z'$ and any $a_1,b_1\ldots,a_n,b_n\in\mathbb R$ such that $S(a_i,b_i)\cup S(a_j,b_j)$ is not a string for each pair $i,j$, the subalgebra $\mathcal B(x)$ acts on $L(a_1,b_1)\otimes\ldots\otimes L(a_n,b_n)$ with simple spectrum.
\end{lettertheorem}

We introduce a Hermitian form on real finite-dimensional representations of $\yt$ that extends the Hermitian form on $L(a,b)$ for $\gll$. This form has been discussed in appendix C in \cite{MTV1}. The generators of subalgebras $\mathcal B(x)$ for $x\in Z'$ act with self-adjoint operators with respect to this form. To prove Theorem \ref{introsimpspec} we use that if a commutative algebra acts on a representation with self-adjoint operators and with a cyclic vector, then it has simple spectrum in this representation.

\subsection{Monodromy problem}
We can construct a branched covering of $Z$ for any representation $\Lab$ satisfying the conditions of Theorem $\ref{introcyclvec}$. Define a subvariety of pairs $\mathcal P$ in $Z\times\mathbb P(\Lab)\supset\{(x,l)\mid x\in Z, \text{ } l \text{ is an eigenline for the action of } \mathcal B(x) \text{ in } \Lab\}$. Then we have the projection map $\pi \colon\mathcal P\to Z$.

For any $x\in Z$ the preimage $\pi^{-1}(x)$ is finite since $\Lab$ has a cyclic vector for $\mathcal B(x)$ for any $x\in Z$ by Theorem \ref{introcyclvec}. Thus this is a finite morphism. By theorem \ref{introsimpspec} there is a point $x\in Z$ for which the number of preimages is maximal and is equal to $\dim\Lab$. Therefore this is a $\dim \Lab$-fold branched covering.

From Theorem \ref{introsimpspec} it follows that the restriction of this map to $Z'\subset Z$ is an unbranched covering. It would be interesting to study the monodromy on this covering.

\subsection{The paper is organized as follows}
In section 2 we introduce the necessary objects and notation.

In section 3 we prove the main technical lemma concerning an analog of the Shapovalov determinant \cite{Sh}.

In section 4 we use the results of Section 3 to show that the Bethe subalgebra corresponding to $e_{12}$ has a cyclic vector in the tensor product of $\gll$ Verma modules.

In section 5 we define families of subalgebras and their closures. We discuss what holds in the case of the family of Bethe subalgebras and prove Theorem \ref{introsubalgfam}.

In section 6 we conclude from the existence of a cyclic vector for $B(e_{12})$ for tensor products of Verma modules that irreducible products of evaluation representations have a cyclic vector (Theorem \ref{introcyclvec}).

In section 7 we define unitarity for representations of $\yt$ and show that irreducible tensor products of evaluation representations have the unitary structure coming from the unitary structure for the action of $\gll$.

In section 8 we combine the results of the previous two sections and show that the subfamily of Bethe subalgebras corresponding to real diagonal matrices has simple spectrum in irreducible products of evaluation representations of $\yt$ (Theorem \ref{introsimpspec}). This subfamily is parametrized by $\mathbb RP^1$ therefore we get a cover of $\mathbb RP^1$ whose fibers are the eigenlines in the representation.

\subsection{Acknowledgements.} We would like to thank L. Rybnikov for many insighful discussions and profound attention to our work. We also would like to thank V. Vologodsky for many helpful discussions. This research was carried out within the HSE University Basic Research Program and funded by the Russian Academic Excellence Project '5-100'.

\section{Preliminaries}
In this section we follow the exposition in \cite{M}.

As we discussed above, the Yangian $\yt$ is an associative algebra generated by the elements $t_{ij}^{(r)}$, $1\leqslant i,j\leqslant 2$, $r\in \mathbb Z_{\geqslant 0}$, $t_{ij}^{(0)}=\delta_{ij}$, which can be combined into formal power series in $u^{-1}$: $$t_{ij}(u)=\sum_{r\geqslant 0}t_{ij}^{(r)}u^{-r},$$ which themselves can be combined into a matrix with values in formal series with coefficients in $\yt$
$$T(u)=\sum_{i,j}e_{ij}\otimes t_{ij}(u)\in \mathrm{End}(\mathbb C^2)\otimes\yt[[u^{-1}]].$$

The defining relations are $$[t_{ij}^{(r+1)},t_{kl}^{(s)}]-[t_{ij}^{(r)},t_{kl}^{(s+1)}]=t_{kj}^{(r)}t_{il}^{(s)}-t_{kj}^{(s)}t_{il}^{(r)}$$
and can be rewritten as $$(u-v)[t_{ij}(u),t_{kl}(v)] = t_{kj}(u)t_{il}(v)-t_{kj}(v)t_{il}(u).$$

For $1\leqslant k<l\leqslant n$ define $$P_{kl}=\sum_{i,j}1\otimes\ldots\otimes 1\otimes e_{ij}\otimes 1\otimes\ldots\otimes 1\otimes e_{ji}\otimes 1\otimes\ldots\otimes 1\in\mathrm{End}(\mathbb C^2)^{\otimes n}$$
where $e_{ij}$ and $e_{ji}$ are correspondingly in the $k$'th and $l$'th positions. We define the $R$-matrix as $$R_{kl}(u)=1^{\otimes n}-u^{-1}P_{kl}.$$

The $R_{ij}$'s satisfy the Yang-Baxter equation: $$R_{12}(u)R_{13}(u+v)R_{23}(v)=R_{23}(u)R_{13}(u+v)R_{12}(v).$$

The defining relations of $\yt$ can be rewritten in the matrix form: let 
$$T_a(u)=\sum_{i,j} 1^{\otimes (a-1)}\otimes e_{ij}\otimes 1^{\otimes (n-a)}\otimes t_{ij}(u)\in \mathrm{End}(\mathbb C^2)^{\otimes n}\otimes\yt$$ 
where $1$ is the identity matrix. Then the relations can be written in the eqiuvalent form, we will call it the RTT relation (we omit tensoring the $R$-matrix by $1\in\yt$) 
$$R(u-v)T_1(u)T_2(v)=T_2(v)T_1(u)R(u-v).$$

For any formal series $f(u)\in 1+u^{-1}\mathbb C[[u^{-1}]]$, for any matrix $G\in GL_2(\mathbb C)$, or for any $c\in\mathbb C$ one can define an automorphism of the Yangian $\yt$:
\begin{equation}
\label{multaut}
    T(u)\mapsto f(u)T(u),
\end{equation}
\begin{equation}
    T(u)\mapsto GT(u)G^{-1}.
\end{equation}
\begin{equation}
\label{shiftaut}
    T(u)\mapsto T(u - c).
\end{equation}

In this text we are considering Bethe subalgebras. They are maximal commutative subalgebras of $\yt$. All of them contain the center of the Yangian which is generated by the coefficients of the quantum determinant $$\operatorname{qdet}T(u)=t_{11}(u)t_{22}(u-1)-t_{21}(u)t_{12}(u-1).$$

Bethe subalgebras are parametrized by elements $C=\left ( \begin{smallmatrix}c_{11}&c_{12}\\c_{21}&c_{22}\end{smallmatrix} \right )\in\mathrm{End}(\mathbb C^2)$ and the corresponding Bethe subalgebra $B(C)$ is generated by the coefficients of $\operatorname{qdet}T(u)$ and the coefficients of $$\operatorname{tr}CT(u)=c_{11}t_{11}(u)+c_{12}t_{21}(u)+c_{21}t_{12}(u)+c_{22}t_{22}(u).$$

We will be using in this text that Bethe subalgebras $B(C)$ are stable under the automorphism (\ref{multaut}) of $\yt$ of multiplication by a formal series $T(u)\mapsto f(u)T(u)$.

Now we will introduce the $\yt$ modules we are discussing in this text.

From the relations and the Poincar\'e-Birkhoff-Witt theorem for $\yt$ it follows that the map $E_{ij}\mapsto t_{ij}^{(1)}$ gives an embedding of the universal enveloping algebra $U(\gll)\subset\yt$. It also follows that the map $t_{ij}(u)\mapsto\delta_{ij}+E_{ij}u^{-1}$ defines a surjective homomorphism $\yt\to U(\gll)$, we call it the evaluation map. It is clear that their composition is identity on $U(\gll)$.

Using the evaluation map, one can define the action of $\yt$ on $U(\gll)$-modules. In this text we are considering finite-dimentional irreducible modules $L(a,b)$, Verma modules $M(a,b)$ and contragredient modules $M^{\vee}(a,b)$, all with highest weight $(a,b)$ such that $a-b\in\mathbb Z_{\geqslant 0}$.

The Yangian $\yt$ is a Hopf algebra with the comultiplication $\Delta\colon\yt\to\yt\otimes\yt$ defined by $$\Delta\colon t_{ij}(u)\mapsto \sum_{k=1}^2 t_{ik}(u)\otimes t_{kj}(u).$$
Hence we can define the action of $\yt$ in the tensor products of evaluation representations.

We will also use another comultiplication $\Delta^{opp}$ on $\yt$ which is a composition of $\Delta$ with the linear operator on $\yt\otimes\yt$ exchanging the two factors:
$$\Delta^{opp}\colon t_{ij}(u)\mapsto\sum_{k=1}^2 t_{kj}(u)\otimes t_{ik}(u)$$

Any irreducible finite-dimensional representation of $\yt$ can be realized as a submodule and as a quotient module of a tensor product of two-dimensional evaluation representations.

We will denote $\sigma^R_{ij}(u)=Flip_{ij}\circ R_{ij}(u)$ the composition of the $R$-matrix and the linear operator exchanging $i$'th and $j$'th tensor factors. The $Flip$ operator and the $R$-matrix commute and the Yang-Baxter equation for $R$-matrices can be rewritten in the form $$\sigma^R_{23}(u)\sigma^R_{12}(u+v)\sigma^R_{23}(v)=\sigma^R_{12}(u)\sigma^R_{23}(u+v)\sigma^R_{12}(v)$$
which gives the braid group relations on $\sigma^R_{ij}$'s.

From the Yang-Baxter equation it follows that $\sigma^R_{12}(u)$ gives a homomorphism of $\yt$ representations
$$\sigma^R_{12}(a_1-a_2)\colon L(a_1,a_1-1)\otimes L(a_2,a_2-1)\to L(a_2,a_2-1)\otimes L(a_1,a_1-1).$$
If $|a_1-a_2|\neq 1$ then it is an isomorphism. If $a_1-a_2=1$, then its kernel is the $3$-dimensional representation $L(a_1,a_1-2)$ and its image is one-dimensional. If $a_1-a_2=-1$, then its kernel is a one-dimensional subrepresentation and its image is $L(a_2,a_2-2)$.

Consider the map $$\sigma^R(a_1,\ldots,a_n)\colon L(a_1,a_1-1)\otimes\ldots\otimes L(a_n,a_n-1)\to L(a_n,a_n-1)\otimes\ldots\otimes L(a_1,a_1-1),$$ $$\sigma^R(a_1,\ldots,a_n)=\prod_{1\leqslant l\leqslant n-1}^{\leftarrow} \prod_{1\leqslant k\leqslant n-l}^{\leftarrow} \sigma^R_{k,k+1}(a_{l}-a_{l+k})=$$
$$
\sigma^R_{12}(a_{n-1}-a_n)\ldots\sigma^R_{n-2,n-1}(a_2-a_n)\ldots\sigma^R_{12}(a_2-a_3)\sigma^R_{n-1,n}(a_1-a_n)\ldots\sigma^R_{23}(a_1-a_3)\sigma^R_{12}(a_1-a_2).
$$

We will be using the following well-known statement.

\begin{proposition}
\label{rephomom}
The map $\sigma^R(a_1,\ldots,a_n)$ is a homomoprhism of $\yt$ representations.
If $a_i=a_1+i-1$ for each $i$, then its image is isomorphic to the evaluation representation $L(a_n,a_n-n)$ up to an automorphism of $\yt$ of multiplication by a formal series.
\end{proposition}

\begin{proof}
First note that
$$\sigma^R(a_1,\ldots,a_n)=\prod_{1\leqslant l\leqslant n-1}^{\leftarrow} \prod_{1\leqslant k\leqslant n-l}^{\leftarrow} \sigma^R_{k,k+1}(a_{l}-a_{l+k})=$$
$$=\prod_{1\leqslant l\leqslant n-1}^{\leftarrow}\prod_{1\leqslant k\leqslant n-l}^{\leftarrow} Flip_{k,k+1}\circ \prod_{1\leqslant l\leqslant n-1}^{\leftarrow}\prod_{1\leqslant k\leqslant n-l}^{\leftarrow} R_{l,l+k}(a_{l}-a_{l+k}).$$

By Proposition 1.6.3 in \cite{M} the map $\sigma^R(a_1,a_1+1,\ldots,a_1+n-1)$ is the symmetrization map of $(\mathbb C^2)^{\otimes n}$ composed with changing the order of factors. Hence its image is the symmetric $n$'th power of the $2$-dimensional space, therefore the dimension of the image of $\sigma^R(a_1,a_1+1,\ldots,a_1+n-1)$ is equal to $\dim L(a_1+n-1,a_1-1)$.

The weight of the highest weight vector in the tensor product corresponds to the same Drinfeld polynomial as the weight of the highest weight vector in $L(a_1+n-1,a_1-1)$. They differ by a twist determined by the automorphism (\ref{multaut}). The highest weight vector is symmetric, hence it lies in the image, so the twist of $L(a_1+n-1,a_1-1)$ by an automorphism (\ref{multaut}) is a subrepresentation of $L(a_n,a_n-1)\otimes\ldots\otimes L(a_1,a_1-1)$ and is isomorphic to the image of the map $\sigma^R(a_1,a_1+1,\ldots,a_1+n-1)$.
\end{proof}

In this text we will consider $\yt$ modules up to twisting by automorphisms of multiplication by a formal series as in (\ref{multaut}), since these automorphisms preserve the subalgebras $B(C)$ for any $C$ and the subalgebras appearing in the limit of the family of Bethe subalgebras.

Throughout the text we will be using a filtration on $\yt$. The filtration $F^{\bullet}\yt$ is defined by $\deg t_{ij}^{(r)}=r$. From the defining relations it is clear that this can be extended to a filtration on $\yt$.

Also we will be using a grading on $\yt$. This grading is given by the adjoint action of $h=t_{11}^{(1)}-t_{22}^{(1)}$, the Cartan element in $\gll$ embedded in $\yt$ as above. This also gives us a grading on the $\yt$ representations where the action of $h$ is locally finite.

The evaluation homomorphism is the identity on the embedded $\gll$. Therefore this grading agrees with the corresponding grading defined by the action of $E_{11}-E_{22}$ on the representations of $\gll$, and the product grading on tensor products of representations is the grading obtained from the diagonal action of $\gll$.

\section{Shapovalov dererminant}
\label{shap}
Here we follow the idea introduced in \cite{Sh}. The results of this section seem to be well-known, but we could not find it in the literature.

Let $A$ be a $\mathbb Z_{\geqslant 0}$-graded algebra without zero-divisors. Let $U$ be a $\mathbb Z_{\geqslant 0}$-graded $A$-module such that the Poincare series of $A$ is equal to the Poincare series of $U$. Pick a basis $a_{mi}$ in $A_m$, the $m$'th graded component of $A$, and a basis $u_{mi}$ of $U_m$ correspondingly. Since Poincare series of $A$ and $U$ are equal, the bases have the same cardinality.

Since $A$ has no zero-divizors, $\dim A_0=1=\dim U_0$. Let $u_0$ be a non-zero element of $U_0$.

Suppose that the action of $A$ on $U$ depends polynomially on the parameters $x_1,\ldots,x_n$. Then we can consider the action of $A\otimes\mathbb C[x_1,\ldots,x_n]$ on $U\otimes\mathbb C[x_1,\ldots,x_n]$. Then we can define the Shapovalov matrix $\mathcal D_m$ that expresses $a_{mi}\cdot u_0$ in terms of $u_{mi}$ with coefficients in $\mathbb C[x_1,\ldots,x_n]$. Let $D_m=\operatorname{det} \mathcal D_m$ be the Shapovalov determinant.

Suppose that for generic values of parameters $U$ is generated by the element $u_0\in U_0$.

\begin{lemma}
\label{shapdetgen}
Suppose $D_m$ is divisible by some linear factor $f\in\mathbb C[x_1,\ldots,x_n]$. Then $D_{m+k}$ is divisible by $f^{\dim A_k}$.
\end{lemma}

\begin{proof}
Denote $K=\mathbb C[x_1,\ldots,x_n]_{(f)}$, the localization of $\mathbb C[x_1,\ldots,x_n]$ at the prime ideal generated by $f$. Then consider $\mathcal D_m$ as a matrix with coefficients in $K$.

Consider $U_m\otimes K$ as a free $K$-module. And consider its free $K$-submodule generated by $\mathcal D_m$
$$\mathcal D_m (U_m\otimes K)=(A_m\otimes K)\cdot u_0.$$

Denote $d_m=\dim_{\mathbb C} A_m=\dim_{\mathbb C}U_m$.

Since $K$ is a principal ideal domain, we can pick generators $w_{1},\ldots,w_{d_m}$ of $U_m\otimes K$ as a $K$-module such that there are $g_1,\ldots,g_{d_m}\in K$ and $g_1 w_{1},\ldots,g_{d_m} w_{d_m}$ generate $\mathcal D_m (U_m\otimes K)$ as a $K$-module such that $\mathcal D_m(w_i)=g_iw_i$. If we consider $\mathcal D_m$ as an operator on $U_m$, then its matrix in these set of generators is diagonal. We can pick unique elements $b_{mi}$ of $A_m\otimes K$ such that $b_{mi}\cdot u_0=g_i w_{i}$.

It follows that $\mathcal D_m$ is equal to a product of two matrices of base change ($\{u_{mi}\}$ to $\{w_{i}\}$ and $\{g_i w_{i}\}$ to $\{u_{mi}\}$) and the diagonal matrix $diag(g_1,\ldots,g_{d_m})$ in the appropriate order.

The matrices of base change are invertible, so their determinants are not divisible by $f$, hence the degree of $f$ that divides the determinant of the diagonal matrix $g_1\cdots g_{d_m}$ is equal to the degree of $f$ that divides $D_m$. Hence one of $g_i$'s is divisible by $f$, suppose that $g_1$ is divisible by $f$.

Now we want to show that $D_{m+k}$ is divisible by $f^{d_k}$.

Let $b_{ki}$, $1\leqslant i\leqslant d_{k}$, be a basis of $A_{k}$ over $\mathbb C$. Then we can complete $b_{ki}b_{m1}$, $1\leqslant i\leqslant d_{k}$, to a set of generators of $A_{m+k}\otimes K$ as a $K$-module as follows. The elements $b_{ki}b_{m1}$ are independent since $A$ has no zero divisors, therefore their images in $A_{m+k}\otimes K/(f)$ are linearly independent over $K/(f)$ and we can complete them to a basis of $A_{m+k}\otimes K/(f)$ over $K/(f)$. Since $K$ is a local ring, the liftings of this completed basis generate $A_{m+k}\otimes K$ as well.

Suppose that $u_{m+k,i}$, $1\leqslant i\leqslant d_{m+k}$ is a set of generators of $U_{m+k}$. Then they are also generators of $U_{m+k}\otimes K$ as a $K$-module.

Then $b_{ki}b_{m1}\cdot u_0=b_{ki}\cdot g_1w_{m1}=g_1b_{ki}\cdot w_{m1}$. Since $A$ acts on $U$ polynomially in terms of $x_i$'s, the expression of $b_{ki}b_{m1}\cdot u_0$ in terms of generators $u_{m+k,j}$ of $U_{m+k}\otimes K$ will be divisible by $f$ for any $i$.

Therefore if we write the matrix $\mathcal D_{m+k}$ in these sets of generators, its first $d_k$ columns will be divisible by $f$. Hence $D_{m+k}$ is divisible by $f^{d_k}$.
\end{proof}

\section{Cyclic vector for Bethe subalgebra of the principal nilpotent in the product of $\gll$ Verma modules}

Throughout this section let $B$ be the Bethe subalgebra of $\yt$ corresponding to $C=e_{12}$. Then $B$ is generated by the coefficients of the series $\operatorname{qdet}T(u)=t_{11}(u)t_{22}(u-1)-t_{21}(u)t_{12}(u-1)$ and $t_{21}(u)$.

We want to show that the highest weight vector of the finite-dimentional representations of $\yt$ is cyclic for $B$. For this we will show that it is cyclic in the product of Verma modules which surjects on the finite-dimensional module.

Let $\underline a=(a_1,\ldots,a_n)$ and $\underline b = (b_1,\ldots,b_n)$. Let $\Dab=M(a_1,b_1)\otimes\ldots\otimes M(a_n,b_n)$ where $M(a,b)$ is the Verma module for $\gll$ with highest weight $(a,b)$ and $\yt$ acts on it via the evaluation morphism.

As we discussed above, the action of $h=t_{11}^{(1)}-t_{22}^{(1)}$ produces a grading on $\Dab$, we shift it so that $\deg v_{a_1,b_1}\otimes\ldots\otimes v_{a_n,b_n}=0$. It is the standard grading on $\Dab$ as a $\gll$ module by the adjoint action of $h$. Let $M_m$ be the component of grading $-2m$ in $\Dab$. (Alternatively, we can consider the grading associatied with the adjoint action of $-h/2$.) Let $v_{a,b}\in M(a,b)$ be the highest weight vector. Let $v_{a,b}^m=(t_{21}^{(1)})^m v_{a,b}$. The degree of $v_{a,b}^{m}$ is equal to $-2m$. We can pick a basis of $M_m$ that consists of all vectors $v_{a_1,b_1}^{m_1}\otimes\ldots\otimes v_{a_n,b_n}^{m_n}$ such that $\sum_{i=1}^n m_i=m$.

We want to show that $v_{a_1,b_1}\otimes\ldots\otimes v_{a_n,b_n}$ is cyclic for the action of $B$ on $\Dab$. In order to do so we will investigate the subspace $V\subset\Dab$ generated from the vector $v_{a_1,b_1}\otimes\ldots\otimes v_{a_n,b_n}$ by the algebra $B$. Since coefficients of $\operatorname{qdet}T(u)$ lie in the center of $\yt$, it suffices to consider the action of the subalgebra of $B$ generated by the coefficients of the series $t_{21}(u)$, so further in this section we will write $B$ for this subalgebra.

To understand what subspace is generated from the highest weight vector, we will use the Shapovalov determinant discussed in section \ref{shap}.

We want to study the subspace $V_m = V \cap M_m$. We know that the only coefficients of $t_{21}(u)$ whose action is non-zero in $\Dab$ are $t_{21}^{(i)}$ for $1\leqslant i\leqslant n$, and that they commute. Therefore $B$ acts on $\Dab$ as a polynomial ring $\mathbb C[t_{21}^{(1)},\ldots,t_{21}^{(n)}]$ and we can think of $B$ as this algebra further on. So $V_m$ is generated by $(t_{21}^{(1)})^{k_1}\ldots (t_{21}^{(n)})^{k_m}(v_{a_1,b_1}\otimes\ldots\otimes v_{a_n,b_n})$ such that $\sum_{i=1}^n k_i=m$. The number of vectors is equal to $\dim M_m$.

Therefore we can construct a matrix that expresses $(t_{21}^{(1)})^{k_1}\ldots (t_{21}^{(n)})^{k_n}(v_{a_1,b_1}\otimes\ldots\otimes v_{a_n,b_n})$ in terms of the basis of $M_m$ we introduced earlier. Its determinant $D_m$ is zero if and only if $V_m$ is a proper subspace of $M_m$.

We will use the generic Verma module $\tilde M(x,y)=U(\gll)\otimes_{U(\mathbb C E_{12}+\mathbb C E_{11}+\mathbb C E_{22})}\mathbb C[x,y]$ where $U(\mathbb C E_{12}+\mathbb C E_{11}+\mathbb C E_{22})$ acts on $\mathbb C[x,y]$ in the following way:
$$E_{12}\cdot 1 = 0$$
$$E_{11}\cdot 1 = x$$
$$E_{22}\cdot 1 = y$$

The usual Verma module $M(a,b)$ is obtained from the generic Verma module by plugging in $x=a$ and $y=b$.

Note that numbers $a$ and $b$ are interchangeable with variables $x$, $y$ in terms of action of $\gll$ so we will abuse the notations and use $a$, $b$ both as variables and numbers.

To study when $D_m$ is equal to zero, we will consider $\tilde M(a_1,b_1)\otimes\ldots\otimes \tilde M(a_n,b_n)$ regarded as a $\mathbb C[a_1,b_1,\ldots,a_n,b_n]$-module with the action of $\yt$ defined on each factor by the evaluation morphism.

Let $\mathcal D_m$ be the matrix expressing the monomial basis consisting of $(t_{21}^{(1)})^{k_1}\ldots (t_{21}^{(n)})^{k_n}(v_{a_1,b_1}\otimes\ldots\otimes v_{a_n,b_n})$ in terms of the basis $v_{a_1,b_1}^{k'_1}\otimes\ldots\otimes v_{a_n,b_n}^{k'_n}$ of $\tilde M(a_1,b_1)\otimes\ldots\otimes \tilde M(a_n,b_n)$. We want to calculate its determinant $D_m$ which is a polynomial in $a_i$'s and $b_i$'s.

\begin{proposition}
\label{shapdetyang}
The determinant $D_m$ is equal to $\prod_{l=0}^{m-1}\prod_{1\leqslant j<i\leqslant n}(a_i-b_j-l)^{\genfrac{(}{)}{0pt}{2}{m+n-l-2}{n-1}}$ up to a constant factor and its degree is equal to
$\operatorname{deg}D_m={{n}\choose{2}} {{n+m-1}\choose{n}}$.
\end{proposition}

Denote $v = v_{a_1,b_1}^{k_1}\otimes\ldots\otimes v_{a_n,b_n}^{k_n}$. Consider
$$t_{21}(u)(v)=\sum_{(i_1,\ldots,i_{n-1})\in\{1,2\}^{n-1}}t_{2i_1}(u)v_{a_1,b_1}^{k_1}\otimes t_{i_1 i_2}(u)v_{a_2,b_2}^{k_2}\otimes\ldots\otimes t_{i_{n-1}1}(u)v_{a_n,b_n}^{k_n}.$$
The coefficient of $u^{-k}$ is equal to $t_{21}^{(k)}(v)$.
Recall that
$$t_{21}(u)(v_{a,b}^{k})=u^{-1}v_{a,b}^{k+1}$$
$$t_{12}(u)(v_{a,b}^{k})=u^{-1}k(a-b-k+1)v_{a,b}^{k-1}$$
$$t_{11}(u)(v_{a,b}^{k})=(1+(a-k) u^{-1})v_{a,b}^{k}$$
$$t_{22}(u)(v_{a,b}^{k})=(1+(b+k)u^{-1})v_{a,b}^{k-1}$$
Hence the degree in $a_i$'s and $b_i$'s of the coefficient in front of $u^{-k}$ is at most $k-1$ since $u^{-1}$ appears with the coefficient of degree $0$ or $1$ in $a_i$'s and $b_i$'s and at least once the degree is zero for $t_{21}(u)$.
Therefore the degrees of all elements of $\tilde M_m$ (the $m$'th graded component of $\tilde M(a,b)$) in the column corresponding to $(t_{21}^{(1)})^{k_1}\ldots (t_{21}^{(n)})^{k_n}(v)$ is at most $\sum_{i=1}^{n}(i-1)k_i$, and hence the degree of the determinant $D_m$ for $\tilde M_m$ is at most the sum of degrees of columns, i.e.

\begin{equation}
    \label{ubound}
    \operatorname{deg} D_m\leqslant\sum_{\substack{(k_1,\ldots,k_n)\\ \sum k_i=m}}(\sum_{i=1}^{n}(i-1)k_i)
\end{equation}

Now we want to estimate the degree of the determinant from below. For this we want to find some vectors that are not generated from $v_{a_1,b_1}\otimes v_{a_2,b_2}$ in $M(a_1,b_1)\otimes M(a_2,b_2)$. By the contragredient duality it is the same as finding singular vectors in the contragredient dual module $M^{\vee}(a_2,b_2)\otimes M^{\vee}(a_1,b_1)$.

\begin{lemma}
\label{singularvect}
For the action of $t_{12}(u)$ on $M^{\vee}(a_2,b_2)\otimes M^{\vee}(a_1,b_1)$ there is one singular vector in degree $m\geqslant 1$ if $a_2-b_1=m-1$ and no singular vectors otherwise.
\end{lemma}

\begin{proof}
Let $w_{a_i,b_i}^{k}$ be the basis vector of the $k$'th degree in $M^{\vee}(a_i,b_i)$ (the grading again comes from the adjoint action of $h$ such that the highest weight vector has degree $0$) such that the action of $t_{12}^{(1)}$ takes $w_{a_i,b_i}^{k}$ to $w_{a_i,b_i}^{k-1}$.

If a vector is singular for $t_{12}(u)$, then it is singular for the diagonal action of $e_{12}\in \sll\subset\gll\subset\yt$. By a direct computation it can be verified that the only such vectors are the scalar multiples of $\sum_{i=0}^{m}(-1)^i w_{a_2,b_2}^{i}\otimes w_{a_1,b_1}^{m-i}$. Let us calculate the action of $t_{12}(u)$ on such vector:
\begin{flalign*}
&t_{12}(u)(\sum_{i=0}^{m}(-1)^i w_{a_2,b_2}^{i}\otimes w_{a_1,b_1}^{m-i})=\\
&\sum_{i=0}^{m}(-1)^i t_{11}(u)w_{a_2,b_2}^{i}\otimes t_{12}(u)w_{a_1,b_1}^{m-i}+\sum_{i=0}^{m}(-1)^i t_{12}(u)w_{a_2,b_2}^{i}\otimes t_{22}(u)w_{a_1,b_1}^{m-i}=\\
&\sum_{i=0}^{m}(-1)^i \Big(u^{-1}\big(1+(a_2-i)u^{-1}\big)w_{a_2,b_2}^{i}\otimes w_{a_1,b_1}^{m-i-1}+u^{-1}\big(1+(b_1+m-i)u^{-1}\big)w_{a_2,b_2}^{i-1}\otimes w_{a_1,b_1}^{m-i}\Big)
\end{flalign*}

This equals zero if and only if $a_2-i=b_1+m-i-1$ which is equivalent to $a_2-b_1=m-1$. 
\end{proof}

\begin{lemma}
\label{detfactors}
The determinant $D_{k+1}$ is divisible by $a_i-b_j-k$ for each $i>j$.
\end{lemma}

\begin{proof}
Consider the module $\Dab$ as above such that $a_l-b_l$ is integer for each $1\leqslant l\leqslant n$, $a_i-b_j=k\in\mathbb Z_{\geqslant 0}$ and $a_l$'s are generic for $l\neq i,j$. We will show that for such module the highest weight vector is not cyclic and $M_{k+1}$ is not generated from the highest weight vector by the action of $B$.
Consider the module $M(\underline a',\underline b')$ such that $a'_l=a_{\sigma(l)}$ and $b'_l=b_{\sigma(l)}$ where $\sigma\in\mathfrak S_n$ is a permutation such that $\sigma(j)=1$, $\sigma(i)=2$.

From Lemma 6.1 in \cite{KT} it follows that for generic values of $c_1,d_1,c_2,d_2$ with $c_1-d_1,c_2-d_2\in\mathbb Z_{\geqslant 0}$, there is a non-zero homomorphism $M(c_1,d_1)\otimes M(c_2,d_2)\to M(c_2,d_2)\otimes M(c_1,d_1)$ mapping the highest weight vector to the highest weight vector. Hence we can construct a morphism $M(\underline a',\underline b')\to\Dab$ as a composition of such morphisms realising the permutation $\sigma$.

From lemma \ref{singularvect} it follows that in $M(\underline a',\underline b')$ the highest weight vector $v_{a'_1,b'_1}\otimes\ldots\otimes v_{a'_n,b'_n}$ is not cyclic and $\dim B_{k+1}v_{a'_1,b'_1}\otimes\ldots\otimes v_{a'_n,b'_n} < \dim M(\underline a',\underline b')_{k+1}$ (here $B_{k+1}$ is the $(k+1)$st graded component of $B$, $\deg t_{21}^{(l)}=1$) since $M(a'_1,b'_1)\otimes M(a'_2,b'_2)$ satisfies the conditions of the lemma. Therefore $\dim B_{k+1}v_{a'_1,b'_1}\otimes\ldots\otimes v_{a'_n,b'_n} < \dim M(\underline a',\underline b')_{k+1}$, and it follows that $\dim B_{k+1}v_{a_1,b_1}\otimes\ldots\otimes v_{a_n,b_n}<M_{k+1}$ and $M_{k+1}$ is not generated from the highest weight vector by the action of $B$.
Hence $D_{k+1}$ is divisible by $a_i-b_j-k$ for each $i>j$.
\end{proof}

\begin{lemma}
\label{freeaction}
For generic values of the parameters $\underline a$, $\underline b$ the highest weight vector $v_{a_1,b_1}\otimes\ldots\otimes v_{a_n,b_n}$ is cyclic for the action of $B$ on $\Dab$.
\end{lemma}

\begin{proof}

We can define the evaluation module $M(a,b)$ by composing the evaluation morphism for the module $M(0,b-a)$ with the shift automorphism (\ref{shiftaut}) and the multiplication autmorphism (\ref{multaut})  $T(u)\mapsto (1+au^{-1})T(u+a)$. Then the twisted action on $M(0,b-a)$ is defined by
\[t_{11}(u)\mapsto 1+u^{-1}(a+e_{11}),\]
\[t_{12}(u)\mapsto u^{-1}e_{12},\]
\[t_{21}(u)\mapsto u^{-1}e_{21},\]
\[t_{22}(u)\mapsto 1+u^{-1}(a+e_{22}),\]
which gives a $\yt$ module isomorphic to $M(a,b)$.

We identify the tensor product of $M(0,b_i-a_i)$ with this twisted action with $\Dab$.
There $t_{21}^{(r)}$ acts by an element of $U(\gll)^{\otimes n}\otimes\mathbb C[a_1,\ldots,a_n]$ that has degree $r-1$ in $a_i$'s.

The vector space of the representation does not depend on $(\underline a,\underline b)$, so we can consider these representations as representations with the same vector space and the operators depend on $(\underline a,\underline b)$.

We can restrict the values of the parameters to the line $(s^{-1}\underline a, \underline b + (s^{-1}-1)\underline a)$ for a fixed $(\underline a, \underline b)$. The operator $t_{21}^{(r)}$ depends on $s$ and has a pole of order $r-1$ at $0$. So operators $s^{r-1}t_{21}^{(r)}$ are well-defined for all $s\in\mathbb C$.

The condition that the highest weight vector in $\Dab$ is cyclic is a Zariski open condition on the space of the parameters. We will show that for the point corresponding to $s=0$ added to the parameter space, the highest weight vector is cyclic. Hence it is cyclic for the parameter values from some Zariski open subset of the projective space. Thus the highest weight vector is cyclic for the values of the parameters from some Zariski open subset of the initial parameter space.

Consider $M(s^{-1}\underline a,\underline b +(s^{-1}-1)\underline a)$ and take the operator $s^{r-1}t_{21}^{(r)}$ at $s=0$. It will be equal to the highest degree (in $a_i$'s) term of the image of $t_{21}^{(r)}$: \[s^{r-1}t_{21}^{(r)}|_{s=0}=\sum_{i=1}^{n}e_{r-1}(a_1,\ldots,\hat{a_i},\ldots,a_n)1^{\otimes i-1}\otimes e_{21}\otimes 1^{\otimes n-i}\]
where $e_r(x_1,\ldots,x_n,\ldots)$ is the elementary symmetric function and $\hat{a}_i$ means omitting $a_i$.

Now to show that the vector $v_{a_1,b_1}\otimes\ldots\otimes v_{a_n,b_n}$ is cyclic for the action of $B$ in the limit module, it suffices to prove that it is cyclic for the collection of operators $\sum_{i=1}^{n}e_{r-1}(a_1,\ldots,\hat{a_i},\ldots,a_n)1^{\otimes i-1}\otimes e_{21}\otimes 1^{\otimes n-i}$, $1\leqslant r\leqslant n$.

To show this we will check that $\sum_{i=1}^{n}e_{r-1}(a_1,\ldots,\hat{a_i},\ldots,a_n)1^{\otimes i-1}\otimes e_{21}\otimes 1^{\otimes n-i}$, $1\leqslant r\leqslant n$, are linearly independent, and thus their linear combinations generate $1^{\otimes i-1}\otimes e_{21}\otimes 1^{\otimes n-i}$, $1\leqslant r\leqslant n$.

So we need to show that the determinant of the matrix $K$ where
\[K_{i,j}=e_{j-1}(a_1,\ldots,\hat a_i,\ldots,a_n)\]
is not zero for genetic values of $a_i$'s. Since
\[e_{j-1}(a_1,\ldots,\hat a_i,\ldots,a_n)=\sum_{k=1}^j (-1)^{j-k}a_i^{j-k}e_{k-1}(a_1,\ldots,a_n),\]
we can present $K$ as a product $K=K'K''$ where $K'_{ij}=(-1)^i e_{j-i}(a_1,\ldots,a_n)$ for $j\geqslant i$, $K'_{ij}=0$ for $j<i$ and $K''_{ij}=a_i^{j-1}$. Therefore \[\det K=\det K'\det K'' =(-1)^{\lfloor\frac{n}{2}\rfloor}\prod_{1\leqslant i<j \leqslant n}(a_i-a_j).\]
\end{proof}

The action of $B$ on $\tilde M(a_1,b_1)\otimes\ldots\otimes \tilde M(a_n,b_n)$ satisfies the conditions of lemma \ref{shapdetgen} since for generic values of parameters the module $\Dab$ is generated from the vector $v_{a_1,b_1}\otimes\ldots\otimes v_{a_n,b_n}$ by Lemma \ref{freeaction}. The dimension $\dim M_m={{m+n-1}\choose{n-1}}$. Therefore combining lemmas \ref{shapdetgen} and \ref{detfactors} we get that $D_m$ is divisible by $(a_i-b_j-k)^{{m+n-k-2}\choose{n-1}}$, $k+1\leqslant m$.

Hence $D_m$ is divisible by 
$$\prod_{k=0}^{m-1}\prod_{1\leqslant j<i\leqslant n}(a_i-b_j-k)^{{m+n-k-2}\choose{n-1}}$$

This gives a lower bound on the degree of the determinant:
\begin{equation}
\label{lbound}
    \operatorname{deg} D_m\geqslant {{n}\choose{2}} \sum_{k=0}^{m-1} {{n+m-k-2}\choose{n-1}} = {{n}\choose{2}} {{n+m-1}\choose{n}}.
\end{equation}

The following lemma shows that the bounds (\ref{ubound}) and (\ref{lbound}) are equal.

\begin{lemma}
$$\sum_{\substack{(k_1,\ldots,k_n)\\ \sum k_i=m}}(\sum_{i=1}^{n}(i-1)k_i)={{n}\choose{2}} {{n+m-1}\choose{n}}$$
\end{lemma}

\begin{proof}
Denote $$S_{m,n} = \sum_{\substack{(k_1,\ldots,k_n)\\ \sum k_i=m}}(\sum_{i=1}^{n}(i-1)k_i)$$
We will prove the statement by induction on $m+n$.
For $n=1$ we have $S_{m,1}=0$ and ${{n}\choose{2}}=0$ hence the statement is correct.

Note that the number of $(k_1,\ldots,k_{n})\in\mathbb Z^{n}_{\geqslant 0}$ such that $\sum_{i=1}^n k_i=m$ and $k_n=j$ is equal to ${{m-j+n-2}\choose{n-2}}$ since it is the number of tuples $(k_1,\ldots,k_{n-1})\in\mathbb Z^{n-1}_{\geqslant 0}$ such that $\sum_{i=1}^{n-1} k_i=m-j$, i.e., the number of monomials of degree $m-j$ in $n-1$ variables. Therefore we can rewrite $S_{m,n}$ as following:
$$\sum_{\substack{(k_1,\ldots,k_n)\\ \sum k_i=m}}(\sum_{i=1}^{n}(i-1)k_i)=\sum_{j=0}^{m}\sum_{\substack{(k_1,\ldots,k_{n-1})\\ \sum k_i=m-j}}(\sum_{i=1}^{n-1}(i-1)k_i) + \sum_{j=0}^{m}(n-1)j{{m-j+n-2}\choose{n-2}}=$$\\
$$= \sum_{j=0}^{m}S_{m-j,n-1} + (n-1)\sum_{j=0}^{m}j{{m-j+n-2}\choose{n-2}}$$

Now we calculate this sum using the induction assumption:

$$\sum_{j=0}^{m}S_{m-j,n-1}=\sum_{j=0}^{m}{{n-1}\choose{2}}{{n+m-j-2}\choose{n-1}}={{n-1}\choose{2}}{{n+m-1}\choose{n}}$$

$$\sum_{j=0}^{m}j{{m-j+n-2}\choose{n-2}}=\sum_{j=0}^{m}{{j}\choose{1}}{{m-j+n-2}\choose{n-2}}={{n+m-1}\choose{n}}$$

Therefore $S_{m,n}={{n-1}\choose{2}}{{n+m-1}\choose{n}}+(n-1){{n+m-1}\choose{n}}={{n}\choose{2}}{{n+m-1}\choose{n}}$.

\end{proof}

\section{Bethe subalgebras and families of algebras}

Let $A$ be an algebra with a filtration $F^{\bullet}A$. Following the \cite{HKRW}, we give the definition of a family of subalgebras. Having a family of subalgebras of an algebra $A$ parametrized by a variety $X$ means that we are given a subalgebra $\mathcal A(x)$, for each $x\in X$, such that for each $N \in \mathbb N$, $d_N := dim(\mathcal A(x) \cap F^N A)$ is independent of $x$, and such that the resulting map
\[
X \to \mathbb G(d_N , F^N A)\]\[
x \mapsto \mathcal A(x) \cap F^N A
\]
is a morphism of algebraic varieties (here $\mathbb G(d, V)$ denotes the Grassmannian of
$d$-dimensional subspaces of $V$).

If we have a family of subalgebras of $A$ parametrized by a variety $X$, we can take the closure of this family in the following way.

Consider Rees algebra $\mathfrak A=\bigoplus_{N\geqslant 0} F^N A$ of the filtered algebra $A$. The multiplication $F^{N}A\otimes F^MA\to F^{N+M}A$ is defined by multiplication in $A$. There is a homomorphism of algebras projection map $p\colon\mathfrak A\to A$ that is defined on $F^N A$ as identity map.

The Rees algebra $\mathfrak{A}$ is an algebra over $
\mathbb C[t]$ with $t$ being the unit in degree $1$. If we set $t=1$, we get the initial algebra $A$, and for $t=0$ we get the associated graded of $A$.

For each $x\in X$ we can lift the subalgebra $\mathcal A(x)$ to a subalgebra $p^{-1}(\mathcal A(x))=\mathfrak A(x)\subset \mathfrak A$ and thus obtain a family of subalgebras of $\mathfrak{A}$. We will define the closure of this family and then take its image in $A$.

For each $N$, let \[Z_N \subset \prod_{N'\leqslant N} \mathbb G(d_{N'}, F^{N'} A)\] be the closure of the image of the product map \[X \to\prod_{N'\leqslant N} \mathbb G(d_{N'} , F^{N'} A).\]
Then there are surjective restriction maps $Z_N \to Z_M$ for any $N > M$. The inverse
limit $Z = \varprojlim Z_N$ is well-defined as a pro-algebraic scheme.
We can project \[Z\subset\prod_{N\geqslant 0}\mathbb G(d_N,F^N A)\] to $\mathbb Gr(d_N,F^N A)$ for each $N\geqslant 0$. Therefore for each $z\in Z$ we get a graded subspace $\mathfrak A(z)$ of $\mathfrak{A}$ with $\dim \mathfrak A(z)_N=d_N$. 

Clearly, this defines a commutative graded subalgebra of $\mathfrak A$  since the conditions of being a subalgebra and of being commutative are closed conditions. It projects to a commutative subalgebra of $A$. 

Now we want to apply this general theory in the case of Bethe subalgebras of $\yt$.

Bethe subalgebra $B(C)$ is a commutative subalgebra of $\yt$ that depends on $C=(c_{ij})\in\gll$. It is generated by the coefficients of the quantum determinant $$\operatorname{qdet}T(u)=t_{11}(u)t_{22}(u-1)-t_{21}(u)t_{12}(u-1)$$ and by the coefficients of $\operatorname{tr}CT(u)=c_{11}t_{11}(u)+c_{12}t_{21}(u)+c_{21}t_{12}(u)+c_{22}t_{22}(u)$ where
\[
T(u)=
  \begin{pmatrix}
    t_{11}(u) & t_{12}(u) \\
    t_{21}(u) & t_{22}(u)
  \end{pmatrix}
  .
\]

These generators are algebraically independent for $C$ not equal to a multiple of $\left (\begin{smallmatrix} 1 & 0\\ 0 & 1 \end{smallmatrix}\right )$. Therefore we have a family of subalgebras of $\yt$ over $\mathbb CP^3-\{\left (\begin{smallmatrix} 1 & 0\\ 0 & 1 \end{smallmatrix}\right )\}$ ($B(C)$ depends on $C$ up to a scalar multiple). Over $\left (\begin{smallmatrix} 1 & 0\\ 0 & 1 \end{smallmatrix}\right )$ it happens that $t^{(1)}_{11}+t^{(1)}_{22}$ is both the coefficient of $\operatorname{qdet}T(u)$ and $t_{11}(u)+t_{22}(u)$. We want to define the closure of this family of subalgebras of $\yt$.

\begin{proposition}
\label{limalg}
The closure of the family of Bethe subalgebras $\mathcal B$ is defined over $Z$, the blow up of $\mathbb CP^3$ at $\left (\begin{smallmatrix} 1 & 0\\ 0 & 1 \end{smallmatrix}\right )$. Algebra $\mathcal B(x)$ over a point of the exceptional divisor $x\in\mathbb CP^2\subset Z$ is generated by $B(\left (\begin{smallmatrix} 1 & 0 \\ 0 & 1 \end{smallmatrix}\right ))$ and an element $t\in\langle t_{12}^{(1)}, t_{21}^{(1)}, t_{11}^{(1)}-t_{22}^{(1)}\rangle\simeq\sll\subset\yt$ such that $x$ represents the line $\mathbb C t$.
\end{proposition}

\begin{proof}
As in Proposition 1.5.2 and Theorem 1.7.5 in \cite{M}, we can consider a filtration on $\yt$ such that $\operatorname{deg}t_{ij}^{(r)}=r-1$. The corresponding associated graded $\mathrm{gr}'\yt$ is isomorphic to $U(\gll)$. The coefficient of $u^{-r}$ in $\operatorname{qdet}T(u)$ has the form $t_{11}^{(r)}+t_{22}^{(r)}$ plus terms of degree less than $r-1$. The coefficient of $u^{-r}$ in $\operatorname{tr} CT(u)$ has the form
$$c_{11}t_{11}^{(r)}+c_{12}t_{21}^{(r)}+c_{21}t_{12}^{(r)}+c_{22}t_{22}^{(r)}.$$
It is clear that their images are algebraically independent in $\mathrm{gr}'\yt$ unless $C$ is a multiple of $\left (\begin{smallmatrix} 1 & 0\\ 0 & 1 \end{smallmatrix}\right )$.

Therefore for any $C$ not equal to a multiple of $\left (\begin{smallmatrix} 1 & 0\\ 0 & 1 \end{smallmatrix}\right )$, $B(C)\cap F^N\yt$ has the same dimension. 


We will construct a family of commutative subalgebras of $\yt$ over $Z$ with constant Poincare series (i.e. independent of $x\in Z$) by explicitly defining the subalgebra for each poing of $Z$. It is clear that such family can be lifted to the corresponding family of commutative subalgebras of the Rees algebra of $\yt$ with constant Poincare series.

If we have an embedding of $Z$ into $\prod_{N\geqslant 0}\mathbb G(d_n, F^N\yt)$ extending the embedding
$$\mathbb CP^3 - \left \{\left [\left (\begin{smallmatrix} 1 & 0 \\ 0 & 1 \end{smallmatrix}\right )\right ]\right \}\to\prod_{N\geqslant 0}\mathbb G(d_n, F^N\yt)$$
then the image of $Z$ is the closure of the image of $\mathbb CP^3 -\left \{ \left [\left (\begin{smallmatrix} 1 & 0 \\ 0 & 1 \end{smallmatrix}\right )\right ]\right\}$ in $\prod_{N\geqslant 0}\mathbb G(d_n, F^N\yt)$ since $Z$ is proper.

The variety $Z=\{(\left [\left (\begin{smallmatrix} x_{11} & x_{12} \\ x_{21} & x_{22} \end{smallmatrix}\right )\right ],[y_0:y_1:y_2])\mid y_1(x_{11}-x_{22})=y_0x_1,y_2x_{12}=y_1x_{21},y_2(x_{11}-x_{22})=y_0x_{21}\}\subset\mathbb CP^3\times\mathbb CP^2$ where $[y_0:y_1:y_2]$ are the coordinates on $\mathbb CP^2$. So if $\left [\left (\begin{smallmatrix} x_{11} & x_{12} \\ x_{21} & x_{22} \end{smallmatrix}\right )\right ]\neq \left [\left (\begin{smallmatrix} 1 & 0 \\ 0 & 1 \end{smallmatrix}\right )\right ]$, $[y_0:y_1:y_2]=[x_{11}-x_{22}:x_{12}:x_{21}]$, and otherwise any point of $\mathbb CP^2$ satisfies the equations.

For the point $x=\left (\left [\left (\begin{smallmatrix} x_{11} & x_{12} \\ x_{21} & x_{22} \end{smallmatrix}\right )\right ],[y_0,y_1,y_2]\right)\in Z$ let the subalgebra $\mathcal B(x)$ be generated by the center of $\yt$, the coefficients of $\operatorname{tr} \left(\begin{smallmatrix} x_{11} & x_{12} \\ x_{21} & x_{22} \end{smallmatrix}\right )T(u)$ and $\frac{1}{2}y_0t_{11}^{(1)}+y_1t_{21}^{(1)}+y_2t_{12}^{(1)}-\frac{1}{2}y_0t_{22}^{(1)}$.

Note that if $\left [\left (\begin{smallmatrix} x_{11} & x_{12} \\ x_{21} & x_{22} \end{smallmatrix}\right )\right ]\neq \left [\left (\begin{smallmatrix} 1 & 0 \\ 0 & 1 \end{smallmatrix}\right )\right ]$, then $\mathcal B(x)=B(C)$ with $C=\left (\begin{smallmatrix} x_{11} & x_{12} \\ x_{21} & x_{22} \end{smallmatrix}\right )$.

If $\left [\left (\begin{smallmatrix} x_{11} & x_{12} \\ x_{21} & x_{22} \end{smallmatrix}\right )\right ]= \left [\left (\begin{smallmatrix} 1 & 0 \\ 0 & 1 \end{smallmatrix}\right )\right ]$, then $\mathcal B(x)$ is generated by $B(\left (\begin{smallmatrix} 1 & 0 \\ 0 & 1 \end{smallmatrix}\right ))$ and $\frac{1}{2}y_0t_{11}^{(1)}+y_1t_{21}^{(1)}+y_2t_{12}^{(1)}-\frac{1}{2}y_0t_{22}^{(1)}$. The latter element does not lie in $B(\left (\begin{smallmatrix} 1 & 0 \\ 0 & 1 \end{smallmatrix}\right ))$. This algebra is commutative since commutativity is a closed condition.

By Proposition 6.1 in \cite{IR1}, the number of algebraically independent elements in $B(\left (\begin{smallmatrix} 1 & 0\\ 0 & 1 \end{smallmatrix}\right ))\cap F^N\yt$ is at least the number of algebraically independent elements in $B(C)\cap F^N\yt$ minus $1$ for $N\geqslant 1$ for generic $C$. Hence by adding an algebraically independent element which lies in $F^1\yt - F^0\yt$, we get the same Poincare series as for a generic Bethe subalgebra.

The algebras $\mathcal B(x)$ are different for different points of $Z$.

Therefore for each $x\in Z$ we obtained a commutative subalgebra $\mathcal B(x)$, the Poicare series of $\mathcal B(x)$ are constant, the regular map $$Z\to\prod_{N\geqslant 0}\mathbb G(d_n, F^N\yt)$$ is an embedding and extends the map
$$\mathbb CP^3 - \left \{\left [\left (\begin{smallmatrix} 1 & 0 \\ 0 & 1 \end{smallmatrix}\right )\right ]\right \}\to\prod_{N\geqslant 0}\mathbb G(d_n, F^N\yt).$$ Hence we obtained the closure of the family of Bethe subalgebras.


\end{proof}

\section{Cyclic vector for Bethe subalgebras}
Let $L(a,b)$ be the evaluation representation of $\yt$ with parameters $a, b\in\mathbb C, a-b\in\mathbb Z_{>0}$. Consider $B(C)\subset \yt$, the Bethe subalgebra corresponding to a matrix $C\in\gll$.

We define a string to be a set $S(a,b)=\{a-1,a-2,\ldots,b+1,b\}$ for $a,b\in\mathbb C$, $a>b$. It is known that the representation $\Lab$ is irreducible if, for any $1\leqslant i<j\leqslant n$, $S(a_i,b_i)\cup S(a_j,b_j)$ is not a string. 

\begin{proposition}
\label{cyclvectnonlimit}
The representation $\Lab$ of $\yt$ such that, for any $1\leqslant i<j\leqslant n$, $S(a_i,b_i)\cup S(a_j,b_j)$ is not a string, has a cyclic vector for the action of $B(C)$ where $C$ is a nonscalar matrix.
\end{proposition}

\begin{proof}
From the condition on parameters of the representations it follows that $\Lab$ is irreducible. If we permute the factors of such representation, we get an isomorphic representation. Therefore we can arrange the factors so that $a_i - b_j \notin \mathbb Z_{\geqslant 0}$ for any pair $i>j$.
 
First, consider $C=e_{12}$. We apply proposition \ref{shapdetyang}. The tensor factors of $\Lab$ are ordered so that the determinant $D_m$ does not vanish for any $m$ and hence the vector $v_{a_1,b_1}\otimes\ldots\otimes v_{a_n,b_n}\in\Dab$ is cyclic for $B(e_{12})$. We have that $\Dab\twoheadrightarrow \Lab$ is a surjection and the image of the vector $v_{a_1,b_1}\otimes\ldots\otimes v_{a_n,b_n}$ in $\Lab$ is cyclic for the action of $B(e_{12})$.

For the general case note that for any semisimple nonscalar matrix $C$ we can choose a matrix $A\in GL(2,\mathbb C)$ and a scalar $\lambda\in\mathbb C$ such that $\lambda A^{-1}CA$ lies in an arbitrary chosen open neighbourhood of $e_{12}$. The condition that a representation has a cyclic vector is open, therefore it follows that the action of Bethe subalgebra $B(A^{-1}CA)$ on $\Lab$ has a cyclic vector under the same conditions on $a_i$'s and $b_i$'s as for $B(e_{12})$.

The Yangian automorphism $T(u)\mapsto AT(u)A^{-1}$ maps $B(C)$ to $B(A^{-1}CA)$. The representation $\Lab$ is finite dimensional, so the action of $\gll$ lifts to the action of $GL(2,\mathbb C)$. Hence conjugation by $A$ gives an isomorphic representation. Therefore $B(C)$ has a cyclic vector in $\Lab$.
\end{proof}

It takes somewhat more effort to prove that the algebras appearing in the limit of the family of Bethe subalgebras act on such representations with a cyclic vector. This has been shown in \cite{MTV2}. We provide an independent proof that follows the argument in \cite{HKRW}.

There is an $\sll$-triple in $\yt$: $f=t^{(1)}_{21}$, $e=t^{(1)}_{12}$ and $h=t^{(1)}_{11}-t^{(1)}_{22}$. The limit subalgebras in $\mathcal B$ are generated by the coefficients of $\operatorname{qdet}T(u)$, $t_{11}(u)+t_{22}(u)$ and one of the elements of this $\sll\subset\yt$.

The grading on $\yt$ defined by the adjoint action of $h$ restricts to a decreasing filtration on $B(\left (\begin{smallmatrix} 1 & 1 \\ 0 & 1 \end{smallmatrix}\right ))$ and a grading on $B(\left (\begin{smallmatrix} 1 & 0 \\ 0 & 1 \end{smallmatrix}\right ))$. Note that the degrees of all coefficients of $\mathrm{qdet}(u)$, $t_{11}(u)$ and $t_{22}(u)$ are zero so $B(\left (\begin{smallmatrix} 1 & 0 \\ 0 & 1 \end{smallmatrix}\right ))$ lies in degree $0$. The generators of $B(\left (\begin{smallmatrix} 1 & 1 \\ 0 & 1 \end{smallmatrix}\right ))$ are non-homogeneous and their highest degree components have degree $0$.

\begin{lemma}
The associated graded with respect to this filtrartion $\operatorname{gr}B(\left (\begin{smallmatrix} 1 & 1 \\ 0 & 1 \end{smallmatrix}\right ))=\langle B(\left (\begin{smallmatrix} 1 & 0 \\ 0 & 1 \end{smallmatrix}\right )),t_{21}^{(1)}\rangle$.
\end{lemma}

\begin{proof}
The algebra $B(\left (\begin{smallmatrix} 1 & 1 \\ 0 & 1 \end{smallmatrix}\right ))$ is generated by the center of $\yt$ (which lies in degree $0$ of $\yt$) and the elements $t_{11}^{(r)}+t_{21}^{(r)}+t_{22}^{(r)}$. The degree $0$ component of $t_{11}^{(r)}+t_{21}^{(r)}+t_{22}^{(r)}$ is $t_{11}^{(r)}+t_{22}^{(r)}$. For $r\geqslant 2$ these elements are algebraically independent and with some set of algebraically independent elements generating the center they comprise a set of independent generators of the algebra $B(\left (\begin{smallmatrix} 1 & 0 \\ 0 & 1 \end{smallmatrix}\right ))$. On the other hand, $t_{11}^{(1)}+t_{22}^{(1)}$ lies in the center, so in the associated graded we will also have $t_{21}^{(1)}$.
\end{proof}

Let $(V,\pi)$ be a finite dimensional representation of $\yt$ such that $B(\left (\begin{smallmatrix} 1 & 1 \\ 0 & 1 \end{smallmatrix}\right ))$ acts on it with a cyclic vector. It is graded by the action of $h$ and this grading agrees with the grading on $B(\left (\begin{smallmatrix} 1 & 0 \\ 0 & 1 \end{smallmatrix}\right ))$ and the filtration on $B(\left (\begin{smallmatrix} 1 & 1 \\ 0 & 1 \end{smallmatrix}\right ))$. The representation can be decomposed into irreducible representations with respect to the $\sll$ action: $V=\bigoplus_{\lambda} V_{\lambda}\otimes W_{\lambda}$ where $V_{\lambda}$ is an irreducible representation of $\sll$ with highest weight $\lambda$ and $W_{\lambda}$ is the multiplicity space. This decomposition can be obtained using the Casimir element $\omega\in B(\left (\begin{smallmatrix} 1 & 0 \\ 0 & 1 \end{smallmatrix}\right ))$, $\omega=ef+fe+\frac{h^2}{2}$. Since $B(\left (\begin{smallmatrix} 1 & 0 \\ 0 & 1 \end{smallmatrix}\right ))$ commutes with the $\sll$, it preserves the summands and acts only on $W_{\lambda}$ in $V_{\lambda}\otimes W_{\lambda}$ for any $\lambda$.

\begin{lemma}
\label{componentcyclvect}
For each $\lambda$ in the decomposition above, $B(\left (\begin{smallmatrix} 1 & 0 \\ 0 & 1 \end{smallmatrix}\right ))$ has a cyclic vector in $W_{\lambda}$.
\end{lemma}

\begin{proof}
Let $\tilde{\omega}$ be a lifting of $\omega$ to $B(\left (\begin{smallmatrix} 1 & 1 \\ 0 & 1 \end{smallmatrix}\right ))$. We have the Jordan decomposition $\pi(\tilde{\omega})=\pi(\tilde{\omega})_s+\pi(\tilde{\omega})_n$ into semisimple and nilponent parts. The representation $V$ is graded as a representation of $\yt$, so we can pick a basis in $V$ compatible with the grading. Then $\pi(\tilde\omega)$ is a lower triangular matrix in this basis. Then $\pi(\tilde{\omega})_n$ is a strictly lower triangular matrix. 
Since the difference $\pi(\tilde{\omega})-\pi(\tilde{\omega}_s)$ is strictly lower triangular in the chosen basis, $\pi(\omega)=\operatorname{gr}\pi(\tilde{\omega})=\operatorname{gr}\pi(\tilde{\omega})_s$. Since $\pi(\tilde{\omega})_s$ and $\pi(\tilde{\omega})_n$ can be expressed polynomially in $\pi(\tilde{\omega})$, there are corresponding elements $\omega_s$ and $\omega_n$ in $B(\left (\begin{smallmatrix} 1 & 1 \\ 0 & 1 \end{smallmatrix}\right ))$ such that $\pi(\tilde\omega)_s=\pi(\tilde\omega_s)$, $\pi(\tilde\omega)_n=\pi(\tilde\omega_n)$.

We have a new decomposition $V=\bigoplus_{\lambda} U_{\lambda}$ into eigenspaces of $\pi(\tilde{\omega}_s)$. We have projectors $\tilde{P}_{\lambda}$ to $U_{\lambda}$ and $P_{\lambda}$ to $V_{\lambda}\otimes W_{\lambda}$. These projectors can be expressed polynomially in terms of $\pi(\tilde\omega_s)$ and $\pi(\omega)$ correspondingly and we can use the same polynomials for $\pi(\tilde\omega_s)$ and $\pi(\omega)$ for a specific $\lambda$. Therefore $\operatorname{gr}\tilde{P}_{\lambda}=P_{\lambda}$ (note that if we compute a value of a polynomial on a lower triangular matrix, then an entry on the diagonal can be computed by applying the polynomial to the corresponding diagonal entry of the initial matrix).

Let $v\in V$ be the cyclic vector for the action of $B(\left (\begin{smallmatrix} 1 & 1 \\ 0 & 1 \end{smallmatrix}\right ))$ on $V$. Then $\tilde P_{\lambda}(v)$ is cyclic in $U_{\lambda}$ for the action of $B(\left (\begin{smallmatrix} 1 & 1 \\ 0 & 1 \end{smallmatrix}\right ))$. Therefore $B(\left (\begin{smallmatrix} 1 & 0 \\ 0 & 1 \end{smallmatrix}\right ))=(\operatorname{gr} B(\left (\begin{smallmatrix} 1 & 0 \\ 0 & 1 \end{smallmatrix}\right )))_0$ generates $v_{\lambda}\otimes W_{\lambda}=(V_{\lambda}\otimes W_{\lambda})^{\lambda}=U_{\lambda}^{\lambda}/U_{\lambda}^{\lambda-1}$ from $P_{\lambda}(v)$.

\end{proof}

\begin{proposition}
\label{cyclvectlimit}
Let $B$ be some algebra in the limit of the family $\mathcal B$ and $V$ as above. Then there is a cyclic vector in $V$ for the action of $B$.
\end{proposition}

\begin{proof}
Algebra $B$ is generated by $B(\left (\begin{smallmatrix} 1 & 0 \\ 0 & 1 \end{smallmatrix}\right ))$ and an element of $\sll$ we discussed above. Then by lemma \ref{componentcyclvect} in each $W_\lambda\otimes V_{\lambda}$ for each $\lambda$ we have a cyclic vector with respect to $B$ since any element of $\sll$ has a cyclic vector in $V_{\lambda}$ for any $\lambda$.
The sum of these vectors will be cyclic with respect to $B$ in $V$ since the Casimir element $\omega$ acts on each $V_{\lambda}$ by a different constant.
\end{proof}

Combining proposition \ref{cyclvectnonlimit} and \ref{cyclvectlimit} we obtain the following result (Theorem \ref{introcyclvec}).

\begin{theorem}
\label{cyclvect}
The representation $\Lab$ of $\yt$ such that, for any $1\leqslant i<j\leqslant n$, $S(a_i,b_i)\cup S(a_j,b_j)$ is not a string, has a cyclic vector for the action of $\mathcal B(x)$ for any $x\in Z$.
\end{theorem}

\section{Unitarity of representations}
In this section we would like to discuss under which conditions $\Lab$ is unitary as a representation of $\yt$. We will consider representations $\Lab$ with $\underline a,\underline b\in\mathbb R^n$. The form that we discuss here is also discussed in Appendix C of \cite{MTV1}.

\begin{definition}
A representation $V$ of $\yt$ is called unitary if there is a positive definite
Hermitian form $\left<\cdot,\cdot\right>$ on $V$ such that for any $v,w\in V$
$\left<t_{ij}(u)v,w\right>=\left<v,t_{ji}(u)w\right>$.
\end{definition}

This is a generalization of the notion of unitarity for $\gll$.

\begin{lemma}
For $a,b\in\mathbb R$ representations $L(a,b)$ have a unitary structure.
\end{lemma}

\begin{proof}
If we consider $L(a,b)$ as a representation of $\gll$, there is a form $\left<\cdot,\cdot\right>$ on $L(a,b)$ such that for $E_{ij}\in\gll$ and any $v,w\in L(a,b)$ we have $\left<E_{ij}v,w\right>=\left<v,E_{ji}w\right>$. The action of $\yt$ on $L(a,b)$ is defined by $t_{ij}(u)=\delta_{ij}+u^{-1}E_{ij}$, hence $\left<t_{ij}(u)v,w\right>=\left<v,t_{ji}(u)w\right>$ holds for this form.
\end{proof}

Suppose $V$ is a representation of $\yt$. Then the action of $\yt$ in the representation is defined by the image of $T(u)=(t_{ij}(u))_{1\leqslant i,j\leqslant 2}$ in $\mathrm{End}(V)\otimes Mat_2(u)$. We want to define an action of $\yt$ on the dual space of $V$. This representation we will denote $V^{\vee}$. To define the action on $V^{\vee}$ we will use the trasposition anti-automorphism $\tau\colon T(u)\mapsto T^t(u)$ of $\yt$. Note that it is not the antipode for the comultiplication $\Delta$!

Since $\Delta\circ\tau=\tau\otimes\tau\circ\Delta^{opp}$, $(V\otimes W)^{\vee}$ is naturally isomorphic to $W^{\vee}\otimes V^{\vee}$.

Let $\varphi\colon V\to V^{\vee}$ be a map to the dual space. Denote the adjoint map $\varphi^*\colon V=V^{\vee\vee}\to V^{\vee}$.

\begin{lemma}
A Hermitian form on $V$ such that $V$ is unitary with respect to this form is equivalent to an antilinear isomorphism $\varphi\colon V\to V^{\vee}$ and $\varphi=\varphi^*$.
\end{lemma}

\begin{proof}
This is true since we define the action of $\yt$ on $V^{\vee}$ with the transposition autmorphism.
\end{proof}

If we have positive definite Hermitian forms on $V$ and $W$ such that $V$ and $W$ are unitary with respect to this form, we can construct a positive definite Hermitian form on $V\otimes W$ that makes this representation unitary. The product of forms on $V$ and $W$ gives us an isomorphism of representations $W\otimes V\to W^{\vee}\otimes V^{\vee}=(V\otimes W)^{\vee}$. If we precompose it with a self-adjoint representation homomorphism $V\otimes W\to W\otimes V$, we obtain a homomorphism $V\otimes W\to (V\otimes W)^{\vee}$ that defines a Hermitian form on $V\otimes W$ and makes the representation $V\otimes W$ unitary.

To make this form on $V\otimes W$ positive definite, we will need the homomorphism $V\otimes W\to W\otimes V\to (V\otimes W)^{\vee}$ to be positive definite as a homomorphism to the dual space.

For $\underline a=(a_1,\ldots,a_n)$ we denote $\underline a^{op}=(a_n,\ldots,a_1)$.

\begin{lemma}
The product of Hermitian forms coming from $\gll$ action on $L(a_i,b_i)$ gives an isomorphism of $\yt$ representations $\Lab^{\vee}\simeq\Labop$.
\end{lemma}

\begin{proof}
The duality does not change the highest weight of the representation, hence $L(a,b)^{\vee}\simeq L(a,b)$ for the evaluation representation $L(a,b)$, and puts the tensor factors in the opposite order.
\end{proof}

It follows that to define a Hermitian form on $\Lab$, we need a homomorphism $\Lab\to\Labop$ which is positive definite and self-adjoint.

First we will consider the case where all $L(a_i,b_i)$ are two-dimensional, and then get to the general case using that every $L(a,b)$ can be realized as a subquotient of
$$L(a,a-1)\otimes L(a-1,a-2)\otimes\ldots\otimes L(b+1,b).$$

In the preliminaries section we have introduced the map 
$$\sigma^R(a_1,\ldots,a_n)\colon L(a_1,a_1-1)\otimes\ldots\otimes L(a_n,a_n-1)\to L(a_n,a_n-1)\otimes\ldots\otimes L(a_1,a_1-1)$$
 $$\sigma^R(a_1,\ldots,a_n)=\prod_{1\leqslant l\leqslant n-1}^{\leftarrow}\prod_{1\leqslant k\leqslant n-l}^{\leftarrow} \sigma^R_{k,k+1}(a_{l}-a_{l+k})$$

\begin{lemma}
\label{selfadj}
The map $\sigma^R(a_1,\ldots,a_n)$ is self-adjoint.
\end{lemma}

\begin{proof}
The map $\sigma^R_{12}(a_1-a_2)\colon L(a_1,a_1-1)\otimes L(a_2,a_2-1)\to L(a_2,a_2-1)\otimes L(a_1,a_1-1)$ is self adjoint since it is a composition of a map with symmetric matrix and the map exchanging the tensor factors.

If we have a composition of two operators $\phi_1\circ\phi_2$, the adjoint map $(\phi_1\circ\phi_2)^*=\phi_2^*\circ\phi_1^*$.

Hence the adjoint map of $\sigma^R(a_1,\ldots,a_n)$ is a map $\Lab\to\Labop$ that is a product of $n(n-1)$ maps $\sigma^R_{ij}(u)$, and they realize the unique permutation of the length $n(n-1)$ on $n$ tensor factors, i.e., the permutation that puts the factors in the opposite order. Also we know that we have the braid group relations on $\sigma^R_{ij}$'s. From the Theorem 3.3.1 in \cite{BB} it follows that any two minimal presentations of an element of $\mathfrak S_n$ can be changed into each other using only braid group relations. Hence $\sigma^R(a_1,\ldots,a_n)=\sigma^R(a_1,\ldots,a_n)^*$.
\end{proof}

Therefore the composition of the maps $\sigma^R(a_1,\ldots,a_n)$ and the map defined by the product of Hermitian forms gives a Hermitian form on $L(a_1,a_1-1)\otimes\ldots\otimes L(a_n,a_n-1)$. In order for this form to be positive definite, we need the map $\sigma^R(a_1,\ldots,a_n)$ to be positive definite. The map $\sigma^R(a_1,\ldots,a_n)$ is self-adjoint hence its eigenvalues are real.

Since $\sigma^R_{ij}(u)=Flip_{ij}\circ R_{ij}(u)$ we can rewrite $$\sigma^R(a_1,\ldots,a_n)=\prod_{1\leqslant l\leqslant n-1}^{\leftarrow} \prod_{1\leqslant k\leqslant n-l}^{\leftarrow} \sigma^R_{k,k+1}(a_{l}-a_{l+k})=$$
$$=\prod_{1\leqslant l\leqslant n-1}^{\leftarrow}\prod_{1\leqslant k\leqslant n-l}^{\leftarrow} Flip_{k,k+1}\circ \prod_{1\leqslant l\leqslant n-1}^{\leftarrow}\prod_{1\leqslant k\leqslant n-l}^{\leftarrow} R_{l,l+k}(a_{l}-a_{l+k}).$$
Therefore it suffices to check that all eigenvalues of 
$$R(a_1,\ldots,a_n)=\prod_{1\leqslant l\leqslant n-1}^{\leftarrow}\prod_{1\leqslant k\leqslant n-l}^{\leftarrow} R_{l,l+k}(a_{l}-a_{l+k})$$
are positive.

\begin{lemma}
\label{poseigval}
The eigenvalues of
$$R(a_1,\ldots,a_n)\colon L(a_1,a_1-1)\otimes\ldots\otimes L(a_n,a_n-1)\to L(a_1,a_1-1)\otimes\ldots\otimes L(a_n,a_n-1)$$
are positive if $|a_i - a_j| > 1$ for any $i,j$.
\end{lemma}

\begin{proof}
Recall that $R_{kl}=1+(a_k-a_l)^{-1}P_{kl}$ where $P_{kl}$ permutes the k'th and l'th coordinates (as defined in preliminaries). All eigenvalues of $P_{kl}$ are $1$ or $-1$ so $R_{kl}$ has zero eigenvalues only if $|a_i - a_j| = 1$.
Now 
$$R(a_1,\ldots,a_n)=\prod_{1\leqslant l\leqslant n-1}^{\leftarrow}\prod_{1\leqslant k\leqslant n-l}^{\leftarrow} R_{l,l+k}(a_{l}-a_{l+k}).$$
Denote $a_{ij}=\frac{1}{a_i - a_j}$. The determinant $$\det(\prod_{1\leqslant l\leqslant n-1}^{\leftarrow}\prod_{1\leqslant k\leqslant n-l}^{\leftarrow} R_{l,l+k}(a_{l}-a_{l+k}))=\prod_{l=1}^{n-1}\prod_{k=1}^{n-l}\det R_{l,l+k}(a_l-a_{l+k})$$
is a polynomial in $a_{ij}$ whose only roots are $a_{ij}=1$ and $a_{ij}=-1$ and if all $a_{ij}=0$ then all eigenvalues are positive. Hence by continuity whenever all $|a_{ij}|<1$, all eigenvalues of $R(a_1,\ldots,a_n)$ are positive.
\end{proof}

Now we can do it in the general case. Let
$$N_<(a,b)=L(b+1,b)\otimes L(b+2,b+1)\otimes\ldots\otimes L(a,a-1)$$
and
$$N_>(a,b)=L(a,a-1)\otimes L(a-1,a-2)\otimes\ldots\otimes L(b+1,b).$$

Let us write an operator
$$\tau\colon N_<(a_1,b_1)\otimes\ldots\otimes N_<(a_n,b_n)\to N_>(a_n,b_n)\otimes\ldots\otimes N_>(a_1,b_1).$$
This operator will be a composition of a map
$$\tau_1\colon N_<(a_1,b_1)\otimes\ldots\otimes N_<(a_n,b_n)\to N_<(a_n,b_n)\otimes\ldots\otimes N_<(a_1,b_1)$$ and a map $$\tau_2\colon N_<(a_n,b_n)\otimes\ldots\otimes N_<(a_1,b_1)\to N_>(a_n,b_n)\otimes\ldots\otimes N_>(a_1,b_1).$$
\\
Denote:\\
$k_i=a_i-b_i$\\
$r_i=k_1+\ldots+k_{i-1}$\\
$q_i=k_{i+1}+\ldots+k_n$\\
$a(s)=a_j$ and $k(s)=s-r_j$ if $r_j \leqslant s < r_{j+1}$.

Define
$$\tau_{1,i}=\prod_{1\leqslant j\leqslant k_i}^{\leftarrow}\prod_{0\leqslant s\leqslant n-r_{i+1}-1}^{\leftarrow} \sigma_{t+s,t+s+1}(b_i+t-a(s+r_{i+1})+k(s+r_{i+1})).$$
Note that $\tau_{1,i}$ pulls $N_<(a_i,b_i)$ through $N_<(a_{i+1},b_{i+1})\otimes\ldots\otimes N_<(a_n,b_n)$.

Then $\tau_1=\prod\limits_{1\leqslant i\leqslant n-1}^{\leftarrow}\tau_{1,i}$.

Define $$\tau_{2,i}=\prod_{1\leqslant j\leqslant k_i-1}^{\leftarrow}\prod_{1\leqslant s\leqslant k_i-j}^{\leftarrow} \sigma_{s+q_i,s+q_i+1}(s).$$
Then $\tau_2=\prod\limits_{1\leqslant i\leqslant n}^{\leftarrow}\tau_{2,i}$.

Define $\tau=\tau_2\tau_1$.

\begin{theorem}
\label{hermform}
Suppose $a_i,b_i\in\mathbb R$ for $1\leqslant i\leqslant n$ and $a_1> b_1>\ldots>a_n>b_n$. Then we can define a Hermitian form on $\Lab$ such that the representation becomes unitary with this form.
\end{theorem}

\begin{proof}
As we have discussed above, to define a unitary form on $\Lab$, we need an isomorphism of $\yt$ representions $\Lab\to \Labop$ that is self-adjoint with respect to the product form defined on $\Lab$ and has positive eigenvalues.

First note that $\tau$ is self-adjoint by the fact that $\sigma_{ij}$ is self-adjoint and lemma \ref{selfadj}. Also note that the eigenvalues of $\tau$ are non-negative since it can be presented as a limit of $R(a_1,\ldots,a_m)$ as in lemma \ref{poseigval} with $a_1>a_2>\ldots>a_m$ and $a_i-a_{i+1}>1$ such that for some $i$'s $a_i-a_{i+1}\to 1$ or $a_i-a_{i+1}\to -1$.

Now we want to show that there is a commutative diagram:

\begin{tikzcd}
N_<(a_1,b_1)\otimes\ldots\otimes N_<(a_n,b_n) \arrow[two heads] {d} {p} \arrow{r} {\tau} & N_>(a_n,b_n)\otimes\ldots\otimes N_>(a_1,b_1)\\
L(a_1,b_1)\otimes\ldots\otimes L(a_n,b_n) \arrow{r}{\tau'} & L(a_n,b_n)\otimes\ldots\otimes L(a_1,b_1) \arrow[u, hook]
\end{tikzcd}

and the map $\tau'$ is the isomorphism we need.

From Proposition \ref{rephomom} it follows that
$$N_<(a_n,b_n)\otimes\ldots\otimes N_<(a_1,b_1)/\ker \tau_{2,i}\simeq N_<(a_n,b_n)\otimes\ldots\otimes L(a_i,b_i)\otimes\ldots\otimes N_<(a_1,b_1),$$ hence $$N_<(a_n,b_n)\otimes\ldots\otimes N_<(a_1,b_1)/\ker \tau_2\simeq L(a_n,b_n)\otimes\ldots\otimes L(a_1,b_1).$$

From this we can see that $\tau$ factors through $\Labop$. Since $\tau_1$ is an isomorphism, $\tau$ is surjective onto $\Labop$.

By lemma \ref{selfadj} we can write $\tau$ as a map that first permutes the factors of $N_<(a_i,b_i)$ and then permutes $N_<(a_i,b_i)$'s. The first map is surjective onto $\Lab$ and the second is an isomorphism. Hence $\tau$ factors through $\Lab$.

Since $\dim \Lab=\dim \Labop$, it follows that the restriction of $\tau$ gives an isomorphism $\Lab\to \Labop$.

The map $\tau=\sigma^R(a_1,a_1-1,\ldots,b_1+1,a_2,\ldots,b_2+1,\ldots,a_n,\ldots,b_n+1)$. Hence by Lemma \ref{selfadj}, $\tau$ is self-adjoint and hence has real eigenvalues. We can view $\tau$ as a limit of maps $\sigma^R(a'_1,\ldots,a'_{n'})$ for $n'=\sum_{i=1}^n(a_i-b_i)$ such that $|a'_i-a'_j|>1$ for any $i,j$. From Lemma \ref{poseigval} it follows that the eigenvalues of $\tau$ are real non-negative. Hence the restriction of $\tau$ that gives the isomorphism $\Lab\to \Labop$ has real positive eigenvalues.
\end{proof}

\section{Simple spectrum and covering}
In this section we consider the closure of the subfamily of $\mathcal B$ corresponding to real diagonal matrices. In the closure there will be one limit algebra, generated by $B(\left (\begin{smallmatrix} 1 & 0 \\ 0 & 1 \end{smallmatrix}\right ))$ and $t^{(1)}_{11}-t^{(1)}_{22}$. This subfamily of $\mathcal B$ corresponds to a subvariety $Z'\subset Z$ which is isomorphic to $\mathbb RP^1$ (closure of the subvariety of real diagonal matrices up to scale).

\begin{theorem}
For any $x\in Z'$, $\mathcal B(x)$ has simple spectrum in $\Lab$ such that $a_i,b_i\in\mathbb R$ for any $i$, and for any $1\leqslant i < j\leqslant n$, $S(a_i,b_i)\cup S(a_j,b_j)$ is not a string.
\end{theorem}

\begin{proof}
By theorem \ref{hermform} we have a Hermitian form on $\Lab$ and the generators of $\mathcal B(x)$ act by self-adjoint operators with respect to this form. Hence the elements of $\mathcal B(x)$ can be all diagonalized in the same basis. By theorem \ref{cyclvect} we have a cyclic vector in $\Lab$ with respect to $\mathcal B(x)$. Hence $\mathcal B(x)$ has simple spectrum in $\Lab$.
\end{proof}

If an algebra $\mathcal A$ acts on a vector space $V$ with simple spectrum, there are $n$ distinct algebra maps $\psi_1, \ldots , \psi_n \colon \mathcal A \to \mathbb C$ such that for each $i$ the eigenspace
$$E_i = \{v \in V \mid a\cdot v = \psi_i(a)v, \text{ for all } a \in\mathcal A\}$$
is one-dimensional. Then $V=E_1 \oplus \ldots \oplus E_n$ and we call $\mathcal E_{\mathcal A}(V):=\{E_1, \ldots, E_n\}$ the set of eigenlines for the action of $\mathcal A$ on $V$.

From the condition that, for any $1\leqslant i < j\leqslant n$, $S(a_i,b_i)\cup S(a_j,b_j)$ is not a string and they are real, it follows that the set of allowed $(a_1,b_1,\ldots,a_n,b_n)$ is a contractible subset $X\subset\mathbb R^{2n}$.

\begin{corol}
Consider $\mathcal E_{\mathcal B(x)}(\Lab)$. For each value of the parameters $(x,a_1,b_1,\ldots,a_n,b_n)\in\mathbb RP^1\times X$ it is a set of $\dim \Lab$ elements that depends smoothly on the parameters. Hence we get an $n$-fold covering of $\mathbb RP^1\times X$.
\end{corol}
 
In the following works we would like to generalize this result to the case of $Y(\mathfrak{gl}_n)$. Also we would like to study the monodromy action of the fundamental group of the base space on this covering.

\bigskip

\footnotesize{
{\bf Inna Mashanova-Golikova} \\
National Research University
Higher School of Economics,\\ Russian Federation,\\
Department of Mathematics, 6 Usacheva st, Moscow 119048;\\
{\tt inna.mashanova@gmail.com}} \\

\end{document}